\documentclass{article}

\usepackage{amsmath} 
\usepackage{amssymb} 
\usepackage{amsthm} 
\usepackage{url}

\newtheorem{theorem}{Theorem}[section]
\newtheorem{lemma}[theorem]{Lemma}
\newcommand{\Gam}{\Gamma}
\newcommand{\dist}{\partial}
\newcommand{\lb}{\big|}

\begin{document}

\title{Sets of Complex Unit Vectors with Two Angles and Distance-Regular Graphs}
\author{Junbo Huang \\ \small{Department of Combinatorics and Optimization} \\ \small{University of Waterloo, Waterloo, Ontario, Canada} \\ \small{j26huang@uwaterloo.ca}}
\date{June 10, 2013}
\maketitle

\begin{abstract}
We study $\{0,\alpha\}$-sets, which are sets of unit vectors of $\mathbb{C}^m$ in which any two distinct vectors have angle 0 or $\alpha$. We investigate some distance-regular graphs that provide new constructions of $\{0,\alpha\}$-sets using a method by Godsil and Roy. We prove bounds for the sizes of $\{0,\alpha\}$-sets of flat vectors, and characterize all the distance-regular graphs that yield $\{0,\alpha\}$-sets meeting the bounds at equality.
\end{abstract}

\section{Introduction}

The \emph{angle} between two unit vectors $x$ and $y$ in $\mathbb{C}^m$ is defined to be the number $|x^*y|^2$, where $x^*$ is the conjugate transpose of $x$. For $\alpha \in \mathbb{R}$ with $0 < \alpha \leq 1$, a set of unit vectors in $\mathbb{C}^m$ is called a \emph{$\{0,\alpha\}$-set} if the angle between any two distinct vectors in the set is 0 or $\alpha$. The concept of $\{0,\alpha\}$-set is motivated by equiangular sets and mutually unbiased bases. A set of unit vectors in $\mathbb{C}^m$ is called \emph{equiangular} if the angle between any two distinct vectors in the set is the same. Two orthonormal bases of $\mathbb{C}^m$ or $\mathbb{R}^m$ are called \emph{unbiased} if the angle between any two vectors from different bases is a constant. Thus equiangular sets and sets of vectors from a collection of mutually unbiased bases are examples of $\{0,\alpha\}$-sets. Equiangular sets and mutually unbiased bases have various applications in quantum theory. For example, in quantum physics, it is desirable to recover the state of a physical system by means of measurements, and equiangular sets and mutually unbiased bases provide ``good'' measurements for such purposes. Because of the applications, equiangular sets and mutually unbiased bases have received much attention. Many constructions for such sets were found; see \cite{CCKS, GodRoy09, Hoggar98, KlapRott04, Konig99}, for examples. Constructions of general $\{0,\alpha\}$-sets, however, received little attention. 

In 2005, Godsil and Roy \cite{RoyThesis} presented a method for constructing sets of unit vectors using certain bipartite graphs. Their method can be used to construct $\{0,\alpha\}$-sets. A vector from a set constructed using Godsil and Roy's method is \emph{flat}, that is, all of its entries have the same absolute value. Flat vectors have natural connections to combinatorics. For example, the columns of a complex Hadamard matrix are flat. In this paper, we investigate some new constructions of flat $\{0,\alpha\}$-sets using Godsil and Roy's method, and study bounds related to such sets. 

We first present examples of distance-regular graphs, including an infinite family, that provide new constructions of $\{0,\alpha\}$-sets using Godsil and Roy's method. We then show that certain upper bounds for the sizes of $\{0,\alpha\}$-sets proven by Delsarte, Goethals and Seidel \cite{DGS} can be improved for sets of flat vectors. We will see that the 8-cycle, the 4-cube, the folded 8-cube, and the coset graph of the extended binary Golay code yield $\{0,\alpha\}$-sets that satisfy the improved bounds at equality. Finally, we show that these are the only distance-regular graphs that provide flat $\{0,\alpha\}$-sets of maximum size with respect to the improved bounds.

\section{Godsil and Roy's Construction}

In this section, we describe a method of constructing sets of unit vectors using certain bipartite graphs. The construction is due to Godsil and Roy, and appeared in Roy's PhD thesis \cite{RoyThesis}. Let $G$ be a group acting on a set $V$. We say that the group $G$ acts on $V$ \emph{regularly} if for every $a, b \in V$, there exists a unique $g \in G$ such that $a^g=b$. A \emph{character} of an abelian group $G$ is a group homomorphism from $G$ to the multiplicative group of complex numbers of norm $1$. It is well known that characters of a finite abelian group $G$ form a group of size $|G|$. For the outline of a proof, see \cite[Sect.~12.8]{AlgCombCG}.

In order to apply Godsil and Roy's construction, we need a connected bipartite graph $\Gam$ with colour classes $Y$ and $Z$ and an abelian group $G$ of automorphisms of $\Gam$ acting regularly on each of $Y$ and $Z$.
Let $u$ and $v$ be vertices of $\Gam$ from the same colour class, and let $h$ be the element in $G$ such that $u^h=v$. Since $h$ is an automorphism of $\Gam$, $h$ defines an injection from the set of neighbours of $u$ to the set of neighbours of $v$. Since the choice of $u$ and $v$ are arbitrary within a colour class, it follows that $\Gam$ must be regular. Let $k$ denote the valency of $\Gam$. Since $G$ acts on $Y$ and $Z$ regularly, we also have $|Y|=|G|=|Z|$. Let us use $n$ to denote the size of $G$, so $\Gam$ has $2n$ vertices.

To start, we pick two vertices $y \in Y$ and $z \in Z$, let
\[
	D := \big\{g \in G: z^g \text{ is adjacent to } y \text{ in } \Gam \big\},
\]
and let $S'$ be the set of characters of $G$ restricted to $D$. Note that $|D|=k$. Since there are $n$ characters of $G$, the set $S'$ is a subset of $\mathbb{C}^k$ and has size $n$. It is shown in \cite[Ch.~4]{RoyThesis} that
\[
	\big\{|x^*y|^2: x,y \in S', x \neq y \big\} = \big\{\lambda^2: \lambda \text{ is an eigenvalue of } \Gam, \lambda \neq \pm k \big\}.
\]
Therefore, by normalizing the vectors in $S'$, we get a set of unit vectors with angle set determined by the eigenvalues of $\Gam$. The construction is summarized in the following lemma.
\begin{lemma}[Godsil and Roy '05]
\label{constr}
Suppose that a connected bipartite graph $\Gam$ with colour classes $Y$ and $Z$ has an abelian group $G$ of automorphisms acting regularly on each of $Y$ and $Z$.
Let $y \in Y$ and $z \in Z$, and let
\[
	D := \big\{g \in G: z^g \text{ is adjacent to } y \text{ in } \Gam \big\}.
\]
Let $S'$ be the set of characters of $G$ restricted to $D$, and let 
\[
	S := \Big\{\frac{1}{\sqrt{k}}\chi: \chi \in S'\Big\},
\]
where $k$ is the valency of $\Gam$. If $2n$ is the number of vertices in $\Gam$, then $S$ is a set of $n$ unit vectors in $\mathbb{C}^k$, with set of angles
\[
	\Big\{\frac{\lambda^2}{k^2}: \lambda \text{ is an eigenvalue of } \Gamma, \lambda \neq \pm k \Big\}.
\]
\end{lemma}\hfill \qedsymbol

In Lemma \ref{constr}, if all the non-identity elements of $G$ have order two, then the range of a character of $G$ is a subset of $\{1,-1\}$; in this case, the set $S$ constructed from $G$ is a subset of $\mathbb{R}^k$. 

When searching for graphs for which Godsil and Roy's construction can be applied, distance-regular graphs are natural candidates. In the next section, we review some basic facts about distance-regular graphs, focusing on those that are bipartite and have diameter four.

\section{Distance-Regular Graphs}

Let $\Gam$ be a graph of diameter $d$ and let $v$ be a vertex of $\Gam$. For any $i \in \mathbb{Z}$, let $\Gam_i(v)$ be the set of vertices of $\Gam$ that are at distance $i$ from $v$. Clearly $\Gam_i(v) = \emptyset$ if $i<0$ or $i>d$. The graph $\Gam$ is called \emph{distance-regular} if it is connected and, for any $i \in \{0,1,\dots,d\}$, the size of $\Gam_i(u) \cap \Gam_1(v)$ depends only on the distance between $u$ and $v$. Let $\dist$ denote the distance function of $\Gam$. If $\dist(u,v)=i$, we write
\[
	a_i := \lb\Gam_i(u) \cap \Gam_1(v)\lb, \quad b_i := \lb\Gam_{i+1}(u) \cap \Gam_1(v)\lb, \quad c_i := \lb\Gam_{i-1}(u) \cap \Gam_1(v)\lb.
\]
The numbers $a_i$, $b_i$ and $c_i$ are called the \emph{intersection numbers} of the distance-regular graph. Since $b_0$ is the valency of any vertex, a distance-regular graph must be regular; we use $k$ to denote the valency $b_0$ of the graph. Let $k_i(v)$ be the size of $\Gam_i(v)$. In particular, $k_0(v)=1$ for each vertex $v$. In fact, by double-counting the number of edges with one end in $\Gam_{i-1}(v)$ and the other in $\Gam_i(v)$, it follows that $k_i(v) c_i = k_{i-1}(v) b_{i-1}$, so inductively we see that $k_i(v)$ is independent of $v$, and we write $k_i$ for $k_i(v)$. The following lemma states some conditions that the intersection numbers satisfy.
\begin{lemma}
\label{feas}
If $b_0$, $b_1$, $\dots$, $b_{d-1}$, $c_1$, $c_2$, $\dots$, $c_d$ are intersection numbers of a distance-regular graph with valency $k$ and diameter $d$, then both of the following conditions hold:
\begin{itemize}
	\item[(i)] $k_i= k_{i-1} b_{i-1}/c_i$ for all $i \in \{1,2,\dots,d\}$.
	\item[(ii)] $1 = c_1 \leq c_2 \leq \dots \leq c_{d-1} \leq c_d \leq k$.
\end{itemize}
\end{lemma}\hfill \qedsymbol

Condition (i) in the lemma has been explained before the lemma. For a proof of condition (ii) , see \cite[Ch.~20]{BiggsAG2}.

The intersection numbers determine many numerical data for the distance-regular graph. For example, from Lemma \ref{feas}, we see that the intersection numbers determine each $k_i$ and hence the number of vertices in the graph. The intersection numbers also determine the eigenvalues of the graph. In fact, the eigenvalues of a distance-regular graph are precisely those of the tridiagonal matrix
\[
B = \begin{bmatrix}
a_0 & c_1 & & & &\\
b_0 & a_1 & c_2 & & & \\
 & b_1 & a_2 & \ddots & & \\
 & & b_2 & \ddots & \ddots & \\
 & & & \ddots & \ddots & c_d \\
 & & & & b_{d-1} & a_d
\end{bmatrix}.
\]
For a proof of this fact, see \cite[Ch.~21]{BiggsAG2}. It is well known that any connected graph of diameter $d$ has at least $d+1$ distinct eigenvalues (see \cite[Ch.~2]{BiggsAG2}), and therefore, by the size of the matrix $B$, any distance-regular graph of diameter $d$ has exactly $d+1$ distinct eigenvalues. For the rest of this section, we focus on bipartite distance-regular graphs of diameter four. For more general theory of distance-regular graphs, we refer the readers to \cite{BCN}. 

Let $\Gam$ be a distance-regular graph with valency $k$. If $\dist(u,v)=i$ then any neighbour of $v$ must be at distance $i-1$, $i$ or $i+1$ from $u$. Consequently,
\[
	\lb \Gam_j(u) \cap \Gam_1(v) \lb = 0,
\]
for $j \notin \{i-1,i,i+1\}$, and so
\[
	a_i + b_i + c_i = k.
\]
If $\Gam$ is bipartite then for any $i$, all vertices in $\Gam_i(v)$ must be in the same colour class, so $a_i=0$ for all $i$, whence
\[
	b_i = k - c_i.
\]
Therefore, for a bipartite distance-regular graph of diameter four, the numbers $k$, $c_2$ and $c_3$ determine all the intersection numbers for the graph. In fact, a bipartite graph of diameter four with $k$, $c_2$ and $c_3$ well-defined is automatically distance-regular. From now on, we use $k$, $c_2$ and $c_3$ to represent all the intersection numbers for a bipartite distance-regular graph of diameter four. For such a graph, the matrix $B$ above becomes
\[
	\begin{bmatrix} 0&1&0&0&0 \\ k&0&c_2&0&0 \\ 0&k-1&0&c_3&0 \\ 0&0&k-c_2&0&k \\ 0&0&0&k-c_3&0 \end{bmatrix}
\]
and has eigenvalues $\pm k$, $\pm \theta_1$ and 0, where $\theta_1 = \sqrt{k + c_2 (k-c_3-1)}$. We call $\theta_1$ the \emph{nontrivial} eigenvalue for a bipartite distance-regular graph of diameter four. For such a graph, Brouwer, Cohen and Neumaier \cite[Thm~5.4.1, p.173]{BCN} proved a restriction on $c_2$ and $c_3$.
\begin{lemma}
\label{c2bound}
Suppose that $c_2$ and $c_3$ are intersection numbers of a distance-regular graph of diameter at least four. If $c_2>1$ then $c_3 \geq 3c_2/2$. 
\end{lemma}\hfill \qedsymbol

We conclude this section by stating an expression of the number of vertices in terms of the intersection numbers.
\begin{lemma}
\label{vertices}
If a bipartite distance-regular graph of diameter four with intersection numbers $k$, $c_2$ and $c_3$ has $2n$ vertices, then
\[
	n = \frac{k(k^2-(c_2+1)k+c_2(c_3 +1))}{c_2 c_3}.
\]
\end{lemma}
\begin{proof}
Since $k_1=k$, by using the formulae in (i) from Lemma \ref{feas}, we deduce that
\[
	k_2 = \frac{k(k-1)}{c_2}, \quad k_3 = \frac{k(k-1)(k-c_2)}{c_2c_3}, \quad k_4 = \frac{(k-1)(k-c_2)(k-c_3)}{c_2c_3}.
\]
Since $2n = k_0+k_1+k_2+k_3+k_4$ and $k_0=1$, we have
\[
	n = \frac{1+k+k_2+k_3+k_4}{2} = \frac{k(k^2-(c_2+1)k+c_2(c_3 +1))}{c_2 c_3}.
\]
\end{proof}

\section{Constructing $\{0,\alpha\}$-Sets}

In this section, we present some distance-regular graphs that can be used to construct $\{0,\alpha\}$-sets using the construction described in Lemma \ref{constr}. By the symmetry of the eigenvalues of a bipartite graph, in order to construction $\{0,\alpha\}$-sets, we need graphs that have exactly five eigenvalues 0, $\pm \theta$ and $\pm k$; distance-regular graphs satisfying this condition are precisely those with diameter four. Bipartite distance-regular graphs of diameter four have been used to construct sets of mutually unbiased bases by Godsil and Roy \cite{GodRoy09}. We describe some examples of such graphs that apply to Lemma \ref{constr}. All the graphs described in this section give new constructions of $\{0,\alpha\}$-sets, except for the 8-cycle, which is a special case of the graphs used in \cite{GodRoy09}.

Let $X$ be a set with an element called ``zero''. The \emph{weight} of an element in $X^m$ is defined to be the number of its nonzero entries. The \emph{Hamming distance} between two elements in $X^m$ is the number of coordinates in which they differ. Let $\mathbb{F}_q$ be the finite field with $q$ elements. Then $\mathbb{F}_q^m$ is a vector space over $\mathbb{F}_q$. Let $C$ be a subspace of $\mathbb{F}_q^m$ such that every nonzero element in $C$ has weight at least two. The \emph{coset graph} of $C$ (with respect to $\mathbb{F}_q^m$) is the graph with vertex set being the set $\mathbb{F}_q^m/C$ of all cosets of $C$, such that two cosets are adjacent if and only if their difference can be represented by a weight-one element in $\mathbb{F}_q^m$.

\subsection{$8$-Cycle}
The $8$-cycle is the graph with vertex set $\mathbb{Z}_2 \times \mathbb{Z}_4$, such that $(x_1,y_1)$ and $(x_2,y_2)$ are adjacent if and only if $x_1 \neq x_2$ and $y_1-y_2 \in \{1,3\}$. The 8-cycle is bipartite, with $\{0\} \times \mathbb{Z}_4$ and $\{1\} \times \mathbb{Z}_4$ being the colour classes, and is the unique bipartite distance-regular graph of diameter four with intersection numbers $k=2$, $c_2=1$ and $c_3=1$. Its nontrivial eigenvalue is $\theta_1=\sqrt{2}$. The group $G:=(\mathbb{Z}_4,+)$ acts on the vertices of the 8-cycle by addition to the second coordinate (without changing the first), and each colour class is an orbit induced by the actions. Since $G$ is clearly a group of automorphisms of the graph, by Lemma \ref{constr}, we can construct a $\{0,1/2\}$-set of size 4 in $\mathbb{C}^2$, represented by the columns of the matrix
\[
	\frac{1}{\sqrt{2}}
	\begin{bmatrix}
	1 & i & 1 & i \\
	1 & -i & i & 1
	\end{bmatrix}.
\]
In fact, these vectors can be partitioned into two unbiased bases of $\mathbb{C}^2$, with a basis containing the first two columns. The construction of vectors from the 8-cycle has essentially been described in \cite{GodRoy09}, in which the 8-cycle is viewed as the incidence graph of the affine plane of order two having the lines with infinite slope removed.

\subsection{$4$-Cube}
The \emph{$4$-cube} is the graph with vertex set $\mathbb{Z}_2^4$, such that two vertices are adjacent if and only if they have Hamming distance 1. It is straightforward to check that the 4-cube is a bipartite distance-regular graph of diameter four having 16 vertices, with intersection numbers $k=4$, $c_2=2$, $c_3=3$ and nontrivial eigenvalue $\theta_1=2$. Moreover, it is the unique bipartite distance-regular graph of diameter four with $k=4$, $c_2=2$ and $c_3=3$; see \cite[Sect.~6.1]{BCN}. The odd-weight elements form a colour class and the even-weight ones form the other. The set of even-weight elements is an additive abelian group $G$ acting regularly on the colour classes by addition, and is clearly a group of automorphisms of the graph. Since $G$ is a subgroup of $\mathbb{Z}_2^4$, every non-identity element of it has order two. Therefore, using Lemma \ref{constr}, we can construct a $\{0,1/4\}$-set of size 8 in $\mathbb{R}^4$, represented by the columns of the matrix
\[
	\frac{1}{2}
	\begin{bmatrix}
	1 & 1 & 1 & 1 & 1 & 1 & 1 & 1 \\
	1 & 1 & -1 & -1 & 1 & 1 & -1 & -1 \\
	1 & -1 & 1 & -1 & 1 & -1 & 1 & -1 \\
	1 & -1 & -1 & 1 & -1 & 1 & 1 & -1
	\end{bmatrix}.
\]
Again, these vectors can be partitioned into two unbiased bases of $\mathbb{R}^4$, with a basis containing the first four columns.

\subsection{Folded $8$-Cube}
The \emph{folded $8$-cube} is the graph with vertex set $\mathbb{Z}_2^7$, such that two vertices are adjacent if any only if they have Hamming distance $1$ or $7$. It is straightforward to check that the folded 8-cube is a bipartite distance-regular graph of diameter four having 128 vertices, with intersection numbers $k=8$, $c_2=2$, $c_3=3$ and nontrivial eigenvalue $\theta_1=4$. Moreoever, it is the unique bipartite distance-regular graph of diameter four with $k=8$, $c_2=2$ and $c_3=3$; see \cite[Sect.~9.2D]{BCN}. The odd-weight elements form a colour class of the even-weight ones form the other. Similar to the 4-cube, the even-weight strings form an abelian group $G$ of graph automorphisms acting regularly on the colour classes by addition. Using Lemma \ref{constr}, we can construct a $\{0,1/4\}$-set of size 64 in $\mathbb{R}^8$ (since $G \leq \mathbb{Z}_2^7$).

\subsection{Van-Lint Schrijver Partial Geometry}

Let $C$ be the subspace of $\mathbb{F}_3^6$ spanned by the all-one vector. Then any coset of $C$ in $\mathbb{F}_3^6$ have its elements sharing the same coordinate sum (computed in $\mathbb{F}_3$). For $i \in \{0,1,2\}$, let $V_i$ be the set of cosets of $C$ whose elements have coordinate sum $i$. It is easy to check that $V_0$, $V_1$ and $V_2$ all have the same size, which is $3^4=81$. Let $\Gam'$ be the coset graph of $C$. Then $\Gam'$ is tripartite, with colour classes $V_0$, $V_1$ and $V_2$. Let $\Gam$ be the subgraph $\Gam$ of $\Gam'$ induced by $V_0$ and $V_1$. The incidence structure with point set $V_0$, line set $V_1$ and incidence graph $\Gam$ is called the \emph{van Lint-Schrijver partial geometry}; it was first introduced in \cite{LintSchrPG}, and is also discussed in \cite{LintCam82} and \cite[Sect.~11.5]{BCN}. The incidence graph $\Gam$ is a bipartite distance-regular graph of diameter four having 162 vertices, with intersection numbers $k=6$, $c_2=1$, $c_3=2$ and nontrivial eigenvalue $\theta_1=3$. See \cite[Sect.~11.5]{BCN}. The set $V_0$ is a subgroup of the abelian group $\mathbb{F}_3^6$; its action on $\mathbb{F}_3^6$ by addition induces orbits $V_0$, $V_1$ and $V_2$, so $V_0$ acts on the colour classes of $\Gam$ regularly. Moreover, addition on the vertices of $\Gam$ by an element of $V_0$ is clearly an automorphism of $\Gam$. Therefore, using Lemma \ref{constr}, we can construct a $\{0,1/4\}$-set of size 81 in $\mathbb{C}^6$.  

\subsection{Extended Binary Golay Code}

Let $I$ be the $12 \times 12$ identity matrix, and let A be the matrix
\[
	\begin{bmatrix} 
	0&1&1&1&1&1&1&1&1&1&1&1 \\ 
	1&1&1&0&1&1&1&0&0&0&1&0 \\
	1&1&0&1&1&1&0&0&0&1&0&1 \\
	1&0&1&1&1&0&0&0&1&0&1&1 \\
	1&1&1&1&0&0&0&1&0&1&1&0 \\
	1&1&1&0&0&0&1&0&1&1&0&1 \\
	1&1&0&0&0&1&0&1&1&0&1&1 \\
	1&0&0&0&1&0&1&1&0&1&1&1 \\
	1&0&0&1&0&1&1&0&1&1&1&0 \\
	1&0&1&0&1&1&0&1&1&1&0&0 \\
	1&1&0&1&1&0&1&1&1&0&0&0 \\
	1&0&1&1&0&1&1&1&0&0&0&1
	\end{bmatrix}.
\]
The \emph{extended binary Golay code} is the subspace of $\mathbb{F}_2^{24}$ generated by the rows of the matrix $[I|A]$. In the extended binary Golay code, any element has even weight (and in fact, weight that is divisible by 4), and any nonzero element has weight at least eight. For more details about Golay codes, see \cite{macsloane77} and \cite{vanLintGTM}. Let $C$ be the extended binary Golay code, and let $\Gam$ be the coset graph of $C$ (with respect to $\mathbb{F}_2^{24}$). Since the elements of $C$ have even weights, the weights of the elements in a coset of $C$ have the same parity; we call a coset of $C$ \emph{even} if its elements have even weights, and \emph{odd} otherwise. It then follows that $\Gam$ is bipartite, with the even cosets forming a colour class and the odd cosets forming the other. In fact, $\Gam$ is a distance-regular graph of diameter four having 4096 vertices, with intersection numbers $k=24$, $c_2=2$, $c_3=3$ and nontrivial eigenvalue $\theta_1 = 8$. Moreoever, $\Gam$ is the unique bipartite distance-regular graph of diameter four with $k=24$, $c_2=2$ and $c_3=3$; see \cite[Sect.~11.3D]{BCN}. The set $G$ of even cosets of $C$ is a subgroup of the quotient group $\mathbb{F}_2^{24}/C$, and it acts on each colour class of $\Gam$ regularly by addition. Note that $G$ acts on $\Gam$ as graph automorphisms. Since the non-identity elements of $G$ have order two, by Lemma \ref{constr}, we can construct a $\{0,1/9\}$-set of size 2048 in $\mathbb{R}^{24}$.

\subsection{Extended Kasami Codes}

Let $s$ and $t$ be powers of 2, with $t \leq s$, and let $F := \mathbb{F}_s$. Let $K(s,t)$ be the set of elements $x$ in $\mathbb{F}_2^F$ with even-weight that satisfy
\[
	\sum\limits_{\alpha \in F}{x_\alpha \alpha} = \sum\limits_{\alpha \in F}{x_\alpha \alpha^{t+1}} = 0.
\]
The subspace $K(s,t)$ of $\mathbb{F}_2^F$ is called the \emph{extended Kasami code} (with parameters $s$ and $t$) if one of the following two conditions is satisfied:
\begin{enumerate}
	\item[(i)] $s=q^{2j+1}$, $t=q^m$, with $q=2^i$, $m \leq j$, and $gcd(m,2j+1)=1$.
	\item[(ii)] $s=q^2$, $t=q$, with $q=2^i$.
\end{enumerate}
Every nonzero element in an extended Kasami code has weight at least four. Similar to the extended Golay code, the elements in a coset of $K(s,t)$ have the same parity, so the coset graph $\Gam(s,t)$ of $K(s,t)$ is bipartite, with the even cosets forming a colour class and the odd ones forming the other. In fact, $\Gam(s,t)$ is a distance-regular graph of diameter four having $2n$ vertices and $k$, $c_2$, $c_3$ as parameters, where
\begin{enumerate}
	\item[(i)] $(n,k,c_2,c_3) = (q^{4j+2},q^{2j+1},q,q^{2j}-1)$,
	\item[(ii)] $(n,k,c_2,c_3) = (q^3,q^2,q,q^2-1)$,
\end{enumerate}
corresponding to the order above. See \cite[Sect.~11.2]{BCN}. The nontrivial eigenvalue of $\Gam(s,t)$ is
\begin{enumerate}
	\item[(i)] $\theta_1=q^{j+1}$.
	\item[(ii)] $\theta_1=q$.
\end{enumerate}
As in the extended binary Golay code, the even cosets of $K(s,t)$ form an abelian group $G$ of automorphisms of $\Gam(s,t)$ that acts on each colour class of $\Gam(s,t)$ regularly by addition. Since every non-identity element of $G$ has order two, by Lemma \ref{constr}, we can construct a $\{0,\alpha\}$-set of size $n$ in $\mathbb{R}^{k}$, where, corresponding to the order above,
\begin{enumerate}
	\item[(i)] $(\alpha,n,k)=(q^{-2j},q^{4j+2},q^{2j+1})$.
	\item[(ii)] $(\alpha,n,k)=(q^{-2},q^3,q^2)$.
\end{enumerate}

\section{Flat Bounds}

In 1975, Delsarte, Goethals and Seidel \cite{DGS} proved bounds for sizes of sets of unit vectors with prescribed sets of angles. One type of the bounds they proved states the following when applied to $\{0,\alpha\}$-sets. 

\begin{theorem}[Delsarte, Goathals and Seidel]
\label{DGS bounds}
Let $S$ be a $\{0,\alpha\}$-set in $\mathbb{C}^m$ with $0<\alpha<1$. Then
\[
	|S| \leq \frac{(m+1)m^2}{2}.
\]
If $S \subseteq \mathbb{R}^m$ then
\[
	|S| \leq \frac{(m+2)(m+1)m}{6}.
\]
\end{theorem}\hfill \qedsymbol

Another proof of this theorem using elementary tensor algebra was later given by Calderbank, Cameron, Kantor and Seidel \cite{CCKS} in 1997. Recall that a vector in $\mathbb{C}^m$ is called \emph{flat} if all of its entries have the same absolute value. It turns out that by using arguments similar to those by Calderbank, Cameron, Kantor and Seidel, the bounds in Theorem \ref{DGS bounds} can be improved if the $\{0,\alpha\}$-set contains only flat vectors.

\begin{theorem}
\label{flatbds}
Let $S$ be a $\{0,\alpha\}$-set of flat vectors in $\mathbb{C}^m$ with $0<\alpha<1$. Then
\[
	|S| \leq \frac{(m^2-m+2)m}{2}.
\]
If $S \subseteq \mathbb{R}^m$ then
\[
	|S| \leq \frac{(m^2-3m+8)m}{6}.
\]
\end{theorem}

\begin{proof}
Let $\beta:=\sqrt{\alpha}$. Then for distinct $x$ and $y$ in $S$, we have $|x^* y| \in \{0,\beta\}$. Let $M$ be the matrix whose columns are the vectors in $S$. Then
\[
	M^* M = I+\beta C,
\]
where $C$ is a Hermitian matrix with zero diagonal that has absolute value 0 or 1 for all off-diagonal entries. For each $x \in S$, let
\[
	v_x := x \otimes x \otimes \overline{x}
\]
be a tensor product, and let $S':=\{v_x: x \in S\}$. For more details about tensor algebra, see \cite{RomanLA}. If we let $N$ be the matrix whose columns are the vectors in $S'$, then since 
\[
	v_x^* v_y = (x \otimes x \otimes \overline{x})^* (y \otimes y \otimes \overline{y}) = (x^* y)^2 (\overline{x^* y}) = (x^* y) |x^* y|^2,
\]
we have
\[
	N^* N = I + \beta^3 C = (1-\beta^2)I + \beta^2(I+\beta C).
\]
Since $|\beta|<1$, the matrix $(1-\beta^2)I$ is positive definite. Since $I+\beta C = M^* M$ is positive semidefinite, so is $\beta^2(I+\beta C)$. Hence $N^* N$ is positive definite and has full rank $|S|$. On the other hand, the rank of $N^* N$ is equal to the rank of $N$, which is the dimension of span$(S')$. Let $x \in S$. Since $x$ is a flat unit vector in $\mathbb{C}^m$, each of its entries has absolute value equal $1/\sqrt{m}$. Consider indices of $v_x$ that have forms $(i,j,j)$ or $(j,i,j)$. The entries of $v_x$ corresponding to these indices are
\[
	x_i x_j \overline{x_j} = \frac{x_i}{m},
\]
which depends only on $x_i$ (since $m$ is a constant). If an index of $v_x$ does not have one of the forms above, then it either has form $(i,i,j)$ or is equal to $(i,j,k)$ for some distinct $i$, $j$ and $k$. There are $m(m-1)$ indices of the first type, and there are $m(m-1)(m-2)$ indices of the second type. Since there are $m$ ways to choose $i$ for the indices $(i,j,j)$ and $(j,i,j)$, and since the entry of $v_x$ indexed by $(i,j,k)$ is equal to that by $(j,i,k)$, it follows that the dimension of span$(S')$ is at most
\[
	m + m(m-1) + \frac{m(m-1)(m-2)}{2} = \frac{m(m^2-m+2)}{2}.
\]
This is an upper bound for the rank of $N^* N$, which is equal to $|S|$, so
\[
	|S| \leq \frac{m(m^2-m+2)}{2},
\]
proving the first bound. Now assume further that $S \subseteq \mathbb{R}^m$. Then any $x \in S$ has entries equal to $1/\sqrt{m}$ or $-1/\sqrt{m}$, and
\[
	v_x = x \otimes x \otimes x.
\]
In this case, the entries of $v_x$ corresponding to the indices of the forms $(i,j,j)$, $(j,i,j)$ or $(j,j,i)$ are all equal to $x_i/m$. Since the entry of $v_x$ indexed by $(i,j,k)$ is invariant under permutations of the components $i$, $j$ and $k$, the dimension of span$(S')$ is at most
\[
	m + \binom{m}{3} = \frac{m(m^2-3m+8)}{6},
\]
proving the second bound
\[
	|S| \leq \frac{m(m^2-3m+8)}{6}.
\]
\end{proof}

Note that Godsil and Roy's construction in Lemma \ref{constr} yields sets of flat vectors, so the bounds proven above are applicable to the $\{0,\alpha\}$-sets constructed using Godsil and Roy's method. In fact, there are distance-regular graphs that produce $\{0,\alpha\}$-sets of optimal sizes with respect to these bounds. In particular, it is easy to check that the 4-cube, the folded 8-cube and the coset graph of the extended binary Golay code produce sets that meet the flat real bound, while the 8-cycle produces a set that meets the flat complex bound, all at equality. The related parameters for these graphs are summarized in Table 1.

\begin{table}[ht]
\caption{$\{0,\alpha\}$-sets of largest size $n$ from graphs.}
\centering
\begin{tabular}{c c c c c c c} 
\hline \noalign{\vskip 0.3ex} 
Graph & $k$ & $c_2$ & $c_3$ & $\alpha$ & $n$ & Space \\ [0.3ex] 
\hline \noalign{\vskip 1ex}
4-cube & 4 & 2 & 3 & 1/4 & 8 & $\mathbb{R}^4$ \\
Folded 8-cube & 8 & 2 & 3 & 1/4 & 64 & $\mathbb{R}^8$ \\
Ext. binary Golay code & 24 & 2 & 3 & 1/9 & 2048 & $\mathbb{R}^{24}$ \\
8-cycle & 2 & 1 & 1 & 1/2 & 4 & $\mathbb{C}^2$ \\[1ex] 
\hline 
\end{tabular}
\end{table}

\section{Optimal Graphs}

At the end of the previous section, we saw some graphs that produce $\{0,\alpha\}$-sets with the largest sizes. In this section, we shall see that those are the only distance-regular graphs that produce largest $\{0,\alpha\}$-sets.

\begin{theorem}
Let $S$ be a $\{0,\alpha\}$-set in $\mathbb{R}^k$ constructed from a distance-regular graph $\Gam$ using the construction in Lemma \ref{constr}. Then $S$ has the largest size $(k^2-3k+8)k/6$ if and only if $\Gam$ is one of the following graphs:
\begin{itemize}
	\item[(i)] $4$-cube.
	\item[(ii)] Folded $8$-cube.
	\item[(iii)] Coset graph of the extended binary Golay code.
\end{itemize}
\end{theorem}

\begin{proof}
The graph $\Gam$ has to be bipartite and have diameter four. Let $k$, $c_2$ and $c_3$ be its intersection numbers, and let $2n$ be the number of its vertices. By Lemma \ref{constr} and Lemma \ref{vertices},
\[
	n = \frac{k(k^2-(c_2+1)k+c_2(c_3 +1))}{c_2 c_3}.
\]
Since $|S|=(k^2-3k+8)k/6$, we have
\[
	\frac{k(k^2-(c_2+1)k+c_2(c_3 +1))}{c_2 c_3} = \frac{(k^2-3k+8)k}{6},
\]
which, after rearrangement, yields
\begin{equation}
\label{realtight2}
	(6-c_2c_3) k^2 - (6c_2+6-3c_2c_3)k + (6c_2-2c_2c_3) = 0.
\end{equation}
Suppose $6-c_2c_3=0$. Then (\ref{realtight2}) becomes
\[
	(6c_2-12)k + (6c_2-12) = 0.
\]
Since $c_2 \leq c_3$ by Lemma \ref{feas}, we have $(c_2,c_3)=(1,6)$ or $(c_2,c_3)=(2,3)$. The case where $c_2=1$ is impossible, since $k=-1$ has to be true for the above equation to hold. On the other hand, $(c_2,c_3)=(2,3)$ satisfies (\ref{realtight2}) no matter what $k$ is equal to. Now suppose $6-c_2c_3 \neq 0$. Applying the quadratic formula to (\ref{realtight2}) yields
\[
	k = \frac{6c_2+6-3c_2c_3 \pm (6c_2-6-c_2c_3)}{12-2c_2c_3}.
\]
The case corresponding to the minus sign is discarded, since in that case $k=1$, which is impossible for a distance-regular graph of diameter four. Hence
\begin{equation}
\label{realtight3}
	k = \frac{6c_2+6-3c_2c_3+6c_2-6-c_2c_3}{12-2c_2c_3} = \frac{6c_2-2c_2c_3}{6-c_2c_3}.
\end{equation}
If $c_2=1$ then rewriting (\ref{realtight3}) yields
\[
	c_3 = \frac{6(k-1)}{k-2}.
\]
Since $k-1$ and $k-2$ are coprime, $k-2$ must divide 6, whence $k \in \{3,4,5,8\}$. Since $c_3<k$, it follows $k=8$, whence $c_3=7$. If $c_2 \geq 2$ then let $\rho:=k/c_2$. By (\ref{realtight3}) we have
\[
	\rho = \frac{6-2c_3}{6-c_2c_3},
\]
which implies
\[
	\frac{\rho c_2 - 2}{\rho - 1} = \frac{6}{c_3}.
\]
Since $\rho-1>0$,
\[
	\frac{6}{c_3} = \frac{\rho c_2 - 2}{\rho - 1} \geq \frac{2\rho - 2}{\rho - 1} = 2,
\]
whence $c_3 \leq 3$. Since $c_2 \geq 2$, Lemma \ref{c2bound} implies that $c_2=2$ and $c_3=3$. But this contradictions our assumption that $6-c_2c_3 \neq 0$.

The arguments above showed that, if $|S|=(k^2-3k+8)k/6$ then one of the following two conditions holds:
\begin{itemize}
	\item[(i)] $c_2=2$ and $c_3=3$.
	\item[(ii)] $k=8$, $c_2=1$ and $c_3=7$.
\end{itemize}
However, graphs with $(k,c_2,c_3)=(8,1,7)$ do not produce $\{0,\alpha\}$-sets in $\mathbb{R}^8$. Indeed, such a graph has nontrivial eigenvalue $\theta_1=\sqrt{8}$ and hence $\alpha=1/8$. On the other hand, any flat unit vectors in $\mathbb{R}^8$ must have entries $\pm 1/\sqrt{8}$, and the only possible angles between such vectors are 9/16, 1/4, 1/16 and 0. So the real flat bound is not applicable to the graphs with $(k,c_2,c_3)=(8,1,7)$. Meanwhile, Brouwer, Cohen and Neumaier \cite[Sect.~4.3, p.153-154]{BCN} showed that any bipartite distance-regular graph of diameter four with $c_2=2$ and $c_3=3$ satisfies $k \in \{4,8,24\}$. These graphs are precisely the 4-cube, the folded 8-cube and the coset graph of the extended binary Golay code.
\end{proof}

The case for $\mathbb{C}^k$ is proven similarly.

\begin{theorem}
Let $S$ be a $\{0,\alpha\}$-set in $\mathbb{C}^k$ constructed from a distance-regular graph $\Gam$ using the construction in Lemma \ref{constr}. Then $S$ has the largest size $(k^2-k+2)k/2$ if and only if $\Gam$ is the 8-cycle.
\end{theorem}

\begin{proof}
The graph $\Gam$ has to be bipartite and have diameter four. Let $k$, $c_2$ and $c_3$ be its intersection numbers. If $|S|=(k^2-k+2)k/2$ then by Lemma \ref{constr} and Lemma \ref{vertices},
\[
	\frac{k(k^2-(c_2+1)k+c_2(c_3 +1))}{c_2 c_3} = \frac{(k^2-k+2)k}{2},
\]
which is equivalent to
\begin{equation}
\label{comptight1}
	(2-c_2c_3) k^2 - (2c_2+2-c_2c_3)k + 2c_2 = 0.
\end{equation}
If $2-c_2c_3=0$ then $c_2=1$, $c_3=2$ and $k=1$, which is impossible, so $2-c_2c_3 \neq 0$. Applying the quadratic formula to (\ref{comptight1}) yields
\[
	k = \frac{2c_2+2-c_2c_3 \pm (2c_2+c_2c_3-2)}{4-2c_2c_3}.
\]
The case with the minus sign is disgarded, for otherwise $k=1$, which is impossible. Therefore
\[
	k = \frac{2c_2+2-c_2c_3+2c_2+c_2c_3-2}{4-2c_2c_3} = \frac{2c_2}{2-c_2c_3}.
\]
The fact that $k$, $c_2$ and $c_3$ are all positive integers then implies that
\[
	c_2 = c_3 = 1 \quad \text{and} \quad k=2.
\]
The unique bipartite distance-regular graph of diameter four with $k=2$, $c_2=1$ and $c_3=1$ is the 8-cycle.
\end{proof}

\noindent \textbf{Acknowledgment} \\

I thank Chris Godsil and Aidan Roy for many helpful discussions on the subject of complex lines.

\bibliographystyle{plain}
\bibliography{mybib}

\end{document}